\documentclass[reqno]{amsart}

\makeatletter
\@namedef{subjclassname@2020}{%
  \textup{2020} Mathematics Subject Classification}
 
\makeatother % Cambiar el ano de la MSC, de 2010 a 2020

\usepackage{amsthm,amssymb,amsfonts,latexsym,mathtools,thmtools,amsmath,lscape}
\usepackage[T1]{fontenc}
\usepackage{tikz-cd} % Commutative diagrams.
\usepackage{enumitem} % Customization of items.
\usepackage{longtable}
\usepackage{array}
\usepackage{multirow}
\usepackage{mathrsfs} 
\usepackage{hyperref} %hyperlinks
\usepackage[all]{xy}
\usepackage{booktabs}

\hypersetup{
    colorlinks=true,
    linkcolor=blue,
    filecolor=blue,      
    urlcolor=blue,
    linktocpage=true
}

\newtheorem{theorem}{Theorem}[section]

\theoremstyle{definition}
\newtheorem{definition}[theorem]{Definition}

\newtheorem{remark}[theorem]{Remark}

\newtheorem{example}[theorem]{Example}

\theoremstyle{remark}

\numberwithin{equation}{section}

\begin{document}

\title[On the category of semi-graded modules]{On the category of semi-graded modules}

\author{Armando Reyes}
\address{Universidad Nacional de Colombia - Sede Bogot\'a}
\curraddr{Campus Universitario}
\email{mareyesv@unal.edu.co}

\thanks{The author was supported by Faculty of Science, Universidad Nacional de Colombia - Sede Bogot\'a, Colombia [grant number 65488].}



\date{}

\dedicatory{Dedicated to my Intellectual Father, Professor Oswaldo Lezama, \\ on the Occasion of His 70th Birthday}

\begin{abstract} 
Lezama \cite{LezamaLatorre2017} introduced the notion of semi-graded ring with the aim of generalizing $\mathbb{Z}$-graded rings and several families of noncommutative rings of polynomial type non-$\mathbb{N}$-graded such as the skew Poincar\'e-Birkhoff-Witt extensions defined by him \cite{GallegoLezama2010}. In a series of papers, \cite{Lezama2020, Lezama2021, LezamaGomez2019, LezamaLatorre2017}, 
he studied problems of non-commutative projective algebraic geometry generalizing the original ideas of Artin et al. \cite{Artin1992, ArtinSchelter1987, ArtinTateVandenBergh2007, ArtinTateVandenBergh1991, ArtinZhang1994} on $\mathbb{N}$-graded rings, in the categorical context of the category $\mathsf{SGR}-R$ of left semi-graded modules over a semi-graded ring $R$. In this note we prove that $\mathsf{SGR}-R$ possesses a canonical set of free generators via shifted twists, which endows the category with a {\em Grothendieck structure} and guarantees the existence of enough injective and projective objects. This categorical robustness allows us to formulate a semi-graded analogue of Baer's criterion for injectivity and to establish a first approach to the dual theory of projective resolutions using shifted twists. % Under this construction, this paper presents a first approach toward a theory of the derived functors $\mathrm{Ext}$ and $\mathrm{Tor}$ in the semi-graded setting ... semi-graded rings  % In particular, he extended the definition of $\mathbb{N}$-graded Artin-Schelter regular algebra to the notion of semi-graded Artin-Schelter regular algebra. Nevertheless, in his definition ...
\end{abstract}

\maketitle

%\tableofcontents 

\section{Introduction}

Lezama \cite{LezamaLatorre2017} introduced the {\em semi-graded rings} as a generalization of $\mathbb{N}$-graded rings and several families of noncommutative rings of polynomial type that are not $\mathbb{N}$-graded, such as {\em Ore extensions} \cite{Fajardoetal2024, Ore1933}, {\em PBW extensions} ({\em PBW} means Poincar\'e-Birkhoff-Witt) \cite{BellGoodearl1988}, {\em 3-dimensional skew polynomial rings} \cite{BellSmith1990, Rosenberg1995}, {\em diffusion algebras} \cite{IsaevPyatovRittenberg2001, PyatovTwarock2002}, {\em skew PBW extensions} \cite{Fajardoetal2020, GallegoLezama2010, LezamaReyes2014} and other families of {\em PBW algebras} \cite{BuesoTorrecillasVerschoren2003, ReyesRodriguez2021, SeilerBook2010}. Since its introduction, Lezama and some researchers have investigated some topics concerning non-commutative projective algebraic geometry of semi graded-rings, with the aim of extending those corresponding ideas formulated by Artin et al. \cite{Artin1992, ArtinSchelter1987, ArtinTateVandenBergh2007, ArtinTateVandenBergh1991, ArtinZhang1994} in the setting of $\mathbb{N}$-graded rings. One of his key results is the generalization of the {\em Serre-Artin-Zhang-Verevkin theorem}, which is known as the non-commutative version of Serre's theorem \cite{Serre1995}. For more details on the subject, see \cite{ArtinZhang1994, ChaconPhDThesis2022, ChaconReyes2025, Lezama2021, Verevkin1992a, Verevkin1992}.

Throughout the research of non-commutative geometry carried out in the above papers, the {\em category of left modules over a semi-graded ring} has been an object of great interest due to its properties and importance in the formulation of different notions (e.g. noncommutative projective scheme, torsion, localization). In particular, the question about the existence of enough injective modules in this category was formulated recently by Ram\'irez in her doctoral thesis \cite[Section 3.4]{RamirezPhDThesis2023} (see also \cite[Section 4]{ChaconRamirezReyes2025}), which arose due to the interest of extending Smith's theory on maps between non-commutative $\mathbb{N}$-graded projective spaces \cite{Smith2003, Smith2016} to the semi-graded setting. 

Precisely, our aim in this is to prove that for a semi-graded ring $R$, the category $\mathsf{SGR}-R$ of left semi-graded modules over $R$ is a Grothendieck category, and hence to guarantee that it possesses enough injective and projective objects. With this result, we are able to formulate a semi-graded analogue of Baer's criterion for injectivity and establish the dual theory of projective resolutions.

 %a canonical set of free generators via shifted twists, which endows the category with a {\em Grothendieck structure} and guarantees the existence of enough injective and projective objects. This categorical robustness allows us 

%In that paper, he considered some notions of noncommutative projective algebraic geometry in the setting of semi-graded rings such as the {\em Hilbert series}, {\em Hilbert polynomial} and {\em Gelfand-Kirillov dimension}. He also extended the notion of {\em noncommutative projective scheme} to the context of semi-graded rings, and generalize  % As a matter of fact, in \cite{Lezama2021, LezamaGomez2019} he computed the set of point modules of finitely semi-graded rings. By considering the parametrization of the point modules for the quantum affine $n$-space, Lezama obtained the set of point modules for some important examples of non-$\mathbb{N}$-graded quantum algebras \cite[Theorem 5.3]{Lezama2020}.

The article is organized as follows. Section \ref{DefinitionsandPreliminaries} contains the key facts on semi-graded rings in order to set up notation and terminology following Lezama's papers \cite{Lezama2021, LezamaLatorre2017}. Section \ref{SectionGrothendieckcategory} outlines the exact construction of the generating set of the category $\mathsf{SGR}-R$ of left semi-graded modules over a semi-graded ring $R$, and proves that this is a Grothendieck category in Theorem \ref{Grothendieckcategory}. We illustrate this result with three classical examples of rings that are not $\mathbb{N}$-graded but are $\mathbb{N}$-semi-graded: the {\em Weyl algebra}, the {\em quantum Weyl algebra}, and the {\em universal enveloping algebra of the finite dimensional Lie algebra} $\mathfrak{sl}_2(\Bbbk)$. Then, Section \ref{Baercriterion} contains a modified version of Baer's criterion for injectivity (Theorem \ref{thm.BaersCriterion}) applied to objects in $\mathsf{SGR}-R$. Finally, Section \ref{Projectiveobjectsandresolutions} presents the dual projective theory on this category: Theorem \ref{thm.EnoughProjectives} establishes that the category $\mathsf{SGR}-R$ has enough projective objects. Then we define the notion of semi-graded global dimension and compute it for the three examples above using shifted twists. %Then, we proceed to give some details about the construction of homogeneous resolutions with some examples. Section \ref{DerivedFunctors} develops the derived functors $\text{Ext}$ and $\text{Tor}$ within $\mathsf{SGR}-R$. As an illustration, we present a first approach toward the theory of global homological dimension, verifying its behavior across three classical $\mathbb{N}$-non-graded noncommutative algebras: the first Weyl algebra $A_1(\Bbbk)$, the quantum Weyl algebra $A_1^q(\Bbbk)$, and the universal enveloping algebra $U(\mathfrak{sl}_2)$ of the Lie algebra $\mathfrak{sl}_2$. 

Throughout the paper, $\mathbb{N}$ denotes the set of natural numbers including zero. The word ring means an associative ring with identity not necessarily commutative, while the symbol $\Bbbk$ denotes a field. The term module will always mean left module unless stated otherwise.

\section{Definitions and preliminaries}\label{DefinitionsandPreliminaries}

\begin{definition}[{\cite[Definition 2.1]{LezamaLatorre2017}}]\label{def.SG2}
Let $R$ be a ring. $R$ is said to be {\em semi-graded} ($\mathsf{SG}$) if there exists a collection $\{R_n\}_{n\in\mathbb{Z}}$ of subgroups $R_n$ of the additive group $R^{+}$ such that the following conditions hold:
\begin{enumerate}
    \item [\rm (i)] $R=\bigoplus\limits_{n\in\mathbb{Z}}R_n$.
    
    \item [\rm (ii)] For every $m,n\in\mathbb{Z}$, we have $R_mR_n\subseteq \bigoplus\limits_{k\leq m+n} R_k$. 
    
    \item [\rm (iii)] $1\in R_0$.
\end{enumerate}
\end{definition}

The collection $\{R_n\}_{n\in\mathbb{Z}}$ is called {\em a semi-graduation of} $R$, and we say that the elements of $R_n$ are {\em homogeneous of degree} $n$. $R$ is {\em positively semi-graded} if $R_n=0$ for every $n<0$. If $R$ and $S$ are semi-graded rings and $f: R\rightarrow S$ is a ring homomorphism, then $f$ is called {\em homogeneous} if $f(R_n)\subseteq S_n$ for every $n\in\mathbb{Z}$. 

All examples mentioned in the Introduction are examples semi-graded rings. For a detailed description of every one, see \cite[Chapters 18 and 19]{Fajardoetal2020} and \cite[Section 3]{LezamaReyes2014}.

\begin{definition}[{\cite[Definition 2.2]{LezamaLatorre2017}}]\label{LezamaLatorre2017Definition2.2}
    Let $R$ be an $\mathsf{SG}$ ring and let $M$ be a left $R$-module. We say that $M$ is  {\em semi-graded} if there exists a collection $\{M_n\}_{n\in\mathbb{Z}}$ of subgroups $M_n$ of the additive group $M^+$ such that the following conditions hold:
    \begin{enumerate}
        \item [\rm (i)] $M=\bigoplus\limits_{n\in\mathbb{Z}}M_n$.
        \item [\rm (ii)] For every $m\geq 0$ and $n\in\mathbb{Z}$, we have $R_mM_n\subseteq \bigoplus\limits_{k\leq m+n}M_k$.
    \end{enumerate}
\end{definition}

The collection $\{M_n\}_{n\in\mathbb{Z}}$ is called {\em a semi-graduation of} $M$, and we say that the elements of $M_n$ are {\em homogeneous of degree} $n$. $M$ is said to be {\em positively semi-graded} if $M_n=0$ for every $n<0$. Let $f: M\rightarrow N$ be a homomorphism of $R$-modules, where $M$ and $N$ are semi-graded $R$-modules. If $f(M_n)\subseteq N_n$ for every $n\in\mathbb{Z}$, then $f$ is called {\em homogeneous}. 

Let $R$ be an $\mathsf{SG}$ ring, $M$ a left $\mathsf{SG}$ $R$-module, and $N$ a submodule of $M$. We say that $N$ is an $\mathsf{SG}$ {\em submodule of} $M$ if $N=\bigoplus\limits_{n\in\mathbb{Z}}N_n$ where $N_n=M_n\cap N$ \cite[Definition 2.3]{LezamaLatorre2017}.

From \cite[Proposition 2.6]{LezamaLatorre2017} we know that if $R$ is an $\mathsf{SG}$ ring, $M$ is a left $\mathsf{SG}$ $R$-module and $N$ is a submodule of $M$, then the following conditions are equivalent:
\begin{enumerate}
        \item [\rm (1)] $N$ is a left $\mathsf{SG}$ submodule of $M$.
        \item [\rm (2)] For every $z\in N$, the homogeneous components of $z$ are in $N$.
        \item [\rm (3)] $M/N$ is a left $\mathsf{SG}$ $R$-module with semi-graduation given by
        \[
        (M/N)_n=(M_n+N)/N,\ n\in\mathbb{Z}.
        \]
\end{enumerate}

% \begin{proof}
%The equivalence between (1) and (2) it is clear.

% ${\rm (2)} \implies {\rm (3)}$ Let $\overline{M}_n:= (M_n+N)/N$, with $n\in\mathbb{Z}$, and $\overline{M}=M/N$. Consider $z\in M$ given by $z = z_1+\dots+z_l$, where $z_k\in M_{n_k}$ for $1\leq k\leq l$. Then $\overline{z}=\overline{z_1}+\dots+\overline{z_l}\in \sum_{n\in\mathbb{Z}}\overline{M}_n$, and so $\overline{M}=\sum_{n\in\mathbb{Z}}\overline{M}_n$. Since the equality $\overline{z_1}+\dots+\overline{z_l}=\overline{0}$ implies $z_1+\dots+z_l\in N$, part (2) shows that $z_k\in N$, i.e., $\overline{z_k} = \overline{0}$ for every $1\leq k\leq l$, and hence the sum is direct. 
 
% Now, let $b_m\in R_m$ and $\overline{z_n}\in\overline{M}_n$. Then $b_m\overline{z_n}=\overline{b_mz_n}=\overline{d_1+\dots+d_p}$, with $d_i\in M_{n_i}$ and $n_i\leq n+m$, whence $\overline{b_m}\overline{z_n}=\overline{d_1}+\dots+\overline{d_p}\in\bigoplus_{k\leq m+n}\overline{M}_k$. Therefore we have showed that $\overline{M}$ is semi-graded.
 
% ${\rm (3)} \implies {\rm (2)}$ Let $z = z_1+\dotsb + z_l\in N$, with $z_i\in M_{n_i}$, and $1\leq i\leq l$. Then $\overline{0} = \overline{z_1}+\dotsb + \overline{z_l}\in\bigoplus_{n\in\mathbb{Z}}\overline{M}_n$, which implies that  $\overline{z_i} = \overline{0}$, and so $z_i\in N$, for every $i$.
% \end{proof}

If $M$ is a left $\mathsf{SG}$ $R$-module and $\{N_i\}_{i\in I}$ is a family of $\mathsf{SG}$ left submodules of $M$, then it is clear that $\bigcap\limits_{i\in I} N_i$ is a left $\mathsf{SG}$ submodule of $M$. 

%Let $X$ be a subset of $M$. We define the $\mathsf{SG}$ left submodule generated by $X$ as the intersection of all $\mathsf{SG}$ left submodules containing $X$, and we will denote it as $\langle X\rangle^{\mathsf{SG}}$. If $X=\{x_1,\dots,x_n\}$, then write $\langle X\rangle^{\mathsf{SG}} = \langle x_1, \dotsc,  x_n\rangle^{\mathsf{SG}}$. We say that $M$ is a {\em left finitely generated} $\mathsf{SG}$ $R$-module if there exist finitely elements $m_1,\dots,m_n$ such that $M = \langle m_1, \dotsc, m_n\rangle^\mathsf{SG}$. If $M$ is simultaneously a left module over different kinds of rings and there is risk of confusion, we write $\langle - \rangle^{\mathsf{SG}}_R$ to indicate the ring $R$ we are considering. Different properties of modules over families of semi-graded rings have been investigated in \cite{Fajardoetal2020, Lezama2020, LouzariReyesSpringer2020, NinoRamirezReyes2020, Reyes2019}.

Examples of semi-graded modules can be found in \cite{Lezama2021, LouzariReyesSpringer2020, NinoRamirezReyes2020, Reyes2019}.

\section{Grothendieck category \texorpdfstring{$\mathsf{SGR}-R$}{Lg}}\label{SectionGrothendieckcategory}

The objective of this section is to show that the category of left semi-graded modules on a semi-graded ring is a {\em Grothendieck category}. In this way and as we said above, we give an answer the question asked by Ram\'irez in her doctoral thesis \cite[Section 3.4]{RamirezPhDThesis2023} and \cite[Section 4]{ChaconRamirezReyes2025}.

Fix an $\mathsf{SG}$ ring $R$ and let $\mathsf{SGR}-R$ be the {\em category of left semi-graded modules over} $R$ where the morphisms are the homogeneous $R-$homomorphisms. It is straightforward to see that $\mathsf{SGR}-R$ is preadditive, and that the zero object of the category is the trivial module. 

Let $f:M\rightarrow N$ be a morphism in $\mathsf{SGR}-R$. Since ${\rm Ker}(f)$ and ${\rm Im}(f)$ are semi-graded submodules, it follows that $N/{\rm Im}(f)$ is a semi-graded module. This fact guarantees that the category $\mathsf{SGR}-R$ has kernels and cokernels. If $f$ is a monomorphism of $\mathsf{SGR}-R$, then $f$ is the kernel of the canonical homomorphism $j:N\rightarrow N/{\rm Im}(f)$. If $f$ is an epimorphism, then $f$ is the cokernel of the inclusion $i:{\rm Ker}(f)\rightarrow M$. In this way, the category $\mathsf{SGR}-R$ is normal and conormal.

If $\{M_i\}_{i\in I}$ is a family of objects of $\mathsf{SGR}-R$, then their direct sum $\bigoplus_{i\in I} M_i$ is a semi-graded ring with semi-graduation given by
\[
\biggl(\bigoplus_{i\in I} M_i\biggr)_p := \bigoplus_{i\in I}(M_i)_p, \ p\in\mathbb{Z}.
\]
It is easy to see that this object with the natural inclusions coincides with the coproduct of the familiy of objects $\{M_i\}_{i\in I}$ in $\mathsf{SGR}-R$. Therefore, $\mathsf{SGR}-R$ is an Abelian category. Details on the subject can be found in Chac\'on's doctoral thesis \cite{ChaconPhDThesis2022}.
%Finally, let $\mathsf{LSG}-R$ be the full subcategory of $\mathsf{SGR}-R$ whose objects are the  $\mathsf{LSG}$ $R$-modules. This subcategory is closed for subobjects, quotients and coproducts, so it is Abelian (see \cite[Section 3.2]{ChaconReyes2025} for more details).

%We denote by $\text{SGR}-S$ the category of left semi-graded modules over $S$, whose morphisms are homogeneous $S$-linear maps of degree zero. 

In his doctoral thesis, Chac\'on proved that $\mathsf{SGR}-R$ satisfies the $\mathsf{Ab5}$ condition (i.e., direct limits exist and are exact functors) \cite[Section 1.5]{ChaconPhDThesis2022}. According to Grothendieck's theorem for Abelian categories in his monumental paper, the \textquotedblleft T{\^o}hoku paper\textquotedblright\ \cite[Théorème 1.10.1.]{Grothendieck1957}, to prove that $\mathsf{SGR}-R$ is a Grothendieck category it suffices to prove the existence of a generating set. With this aim, we consider the following two definitions in the setting of semi-graded rings. 

\begin{definition}\label{def.freeshifts}
Let $R = \bigoplus_{n \in \mathbb{Z}} R_n$ be a fixed $\mathsf{SG}$ ring $R$. For each integer $n \in \mathbb{Z}$, we define the {\em free semi-graded shift} (or {\em twist}) $R(n)$ as the left semi-graded $R$-module whose underlying left module structure is the ring $R$ itself, equipped with the shifted semi-graduation given component-wise by:
\[
(R(n))_k := R_{n + k}, \quad \text{for all } k \in \mathbb{Z}.
\]
In particular, the ring identity element $1_R \in R_0$ resides inside the homogeneous component of degree $-n$ of the shifted object, since $(R(n))_{-n} = R_{n-n} = R_0$.
\end{definition}

\begin{definition}\label{def.SGRgenerators}
Let $R = \bigoplus_{n \in \mathbb{Z}} R_n$ be a fixed $\mathsf{SG}$ ring $R$, and let $\mathsf{SGR}-R$ be the category of left semi-graded $R$-modules. A family of left $\mathsf{SG}$ modules $\left\{G_\lambda\right\}_{\lambda \in \Lambda}$ in $\mathsf{SGR}-R$ is called a {\em set of generators} for the category if it satisfies the following conditions:
\begin{enumerate}
    \item [\rm (i)] {\em Homogeneous element detection.} For every left semi-graded $R$-module $M \in \mathsf{SGR}-R$ and every non-zero homogeneous element $m \in M_j$ of a fixed degree $j \in \mathbb{Z}$, there exists an index $\lambda \in \Lambda$ and a morphism in the category $\phi: G_\lambda \longrightarrow M$ such that $m$ is contained in the image of $\phi$. Since $\phi$ is a morphism in $\mathsf{SGR}-R$, it must be a strict $R$-linear map of degree zero, i.e., 
    \[
    \phi\left((G_\lambda)_k\right) \subseteq M_k, \quad {\rm for \ all } \ k \in \mathbb{Z}.
    \]

    \item [\rm (ii)] {\em Local coverage.} The previous condition implies that every homogeneous component $M_k$ of any module $M$ can be written as the sum of the strict components of the images of all available generating morphisms:
    \[
    M_k = \sum_{\lambda \in \Lambda} \sum_{\phi \in \text{Hom}_{\mathsf{SGR}-R}(G_\lambda, M)} \phi\left((G_\lambda)_k\right).
    \]

    \item [\rm (iii)] {\em Categorical epimorphism.} The family $\{G_\lambda\}_{\lambda \in \Lambda}$ generates the category if for every module $M \in \mathsf{SGR}-R$, there exists a choice of indices $\lambda_i \in \Lambda$ such that $M$ is the image of a surjective homogeneous morphism of degree zero from their direct sum 
    \[
    \bigoplus_{i \in I} G_{\lambda_i} \xrightarrow{\quad \Phi \quad} M \longrightarrow 0,
    \]
    where the direct sum carries the standard canonical semi-graduation \linebreak $\left(\bigoplus I_i\right)_k = \bigoplus (I_i)_k$.
\end{enumerate}
\end{definition}

The following is the first important result of the paper.

\begin{theorem}[Grothendieck category]\label{Grothendieckcategory}
The family of objects $\left\{R(n)\right\}_{n \in \mathbb{Z}}$ constitutes a set of generators for the category $\mathsf{SGR}-R$. Consequently, $\mathsf{SGR}-R$ is a Grothendieck category and possesses enough injective objects.
\end{theorem}
\begin{proof}
To prove that $\left\{R(n)\right\}_{n \in \mathbb{Z}}$ is a set of generators for $\mathsf{SGR}-R$, we show that it satisfies the homogeneous element detection property of Definition \ref{def.SGRgenerators} (i). 
Let $M \in \mathsf{SGR}-R$. Since $M = \bigoplus_{k \in \mathbb{Z}} M_k$, any element in $M$ decomposes uniquely into a finite sum of homogeneous components. Thus, it suffices to generate an arbitrary non-zero homogeneous element $m \in M_j$ for a fixed degree $j \in \mathbb{Z}$.

Consider the free semi-graded shift object $R(-j)$. By Definition \ref{def.freeshifts}, its homogeneous components are given by $(R(-j))_k = R_{-j+k}$. Setting $k = j$, we find $(R(-j))_j = R_{-j+j} = R_0$, which implies that the identity element $1_R$ resides in degree $j$.

We define a map $\bar{\phi}: R(-j) \to M$ component-wise on homogeneous elements. For any element $r \in (R(-j))_k = R_{k-j}$, the definition of the semi-graded module multiplication on $M$ implies that 
\[ 
R_{k-j} \cdot M_j \subseteq \bigoplus_{i \leq (k-j)+j} M_i = \bigoplus_{i \leq k} M_i. 
\]
Thus, the product $r \cdot m$ can be uniquely decomposed into homogeneous components as $r \cdot m = \sum_{i \leq k} (r \cdot m)_i$, where $(r \cdot m)_i \in M_i$. We set the assignment 
\[
\bar{\phi}(r) := (r \cdot m)_k,
\]
which isolates the strict homogeneous component of degree $k$. Extending by linearity to all of $R(-j)$, we verify that $\bar{\phi}$ is a well-defined morphism in $\mathsf{SGR}-R$:
\begin{itemize}
\item Homogeneity of degree zero: By construction, if $r \in (R(-j))_k$, then $\bar{\phi}(r) = (r \cdot m)_k \in M_k$. Thus, $\bar{\phi}$ maps $(R(-j))_k$ into $M_k$ for all $k \in \mathbb{Z}$, which shows that it is a strict morphism of degree zero.

\item $R$-linearity: The filtration condition on the semi-graded multiplication ensures that for any $r_a \in R_a$ and $r_k \in (R(-j))_k$, the leading term of $(r_a r_k) \cdot m$ in degree $a+k$ comes exactly from the left action of $r_a$ on the leading term of $r_k \cdot m$ in degree $k$. Hence, $\bar{\phi}(r_a r_k) = (r_a r_k \cdot m)_{a+k} = r_a \cdot (r_k \cdot m)_k = r_a \bar{\phi}(r_k)$, satisfying the left $R$-module axioms.
\end{itemize}

Evaluating this morphism at the ring identity $1_R \in (R(-j))_j$ yields
\[
\bar{\phi}(1_R) = (1_R \cdot m)_j = m \in M_j,
\]
which proves that $m \in \text{Im}(\bar{\phi})$. 

Since every homogeneous element of $M$ lies in the image of a morphism originating from some member of the family $\left\{R(n)\right\}_{n \in \mathbb{Z}}$, we can sum over all homogeneous elements to construct a global surjective morphism $\Phi: \bigoplus_{\lambda} R(n_{\lambda}) \to M \to 0$. 

Finally, since the category $\mathsf{SGR}-R$ is Abelian, has exact filtered colimits, and possesses the generating set $\left\{R(n)\right\}_{n \in \mathbb{Z}}$, it satisfies Grothendieck's criterion \cite[Théorème 1.10.1]{Grothendieck1957}. Therefore, $\mathsf{SGR}-R$ is a Grothendieck category, which immediately implies it has enough injective objects.
\end{proof}

\begin{remark}
Theorem \ref{Grothendieckcategory} emphasizes that the category $\mathsf{SGR}-R$ has an intrinsic geometric existence completely independent of any localization theory. Unlike previous approaches which limited the foundational definitions to the subcategory $\mathsf{LSG}-R$ of {\em localizable left semi-graded modules over} $R$ whose morphisms are homogeneous $R$-linear maps of degree zero (e.g. \cite{ChaconPhDThesis2022, ChaconRamirezReyes2025, ChaconReyesJAA2025, ChaconReyes2025, RamirezPhDThesis2023}), this new framework ensures that enough injectives exist globally in the semi-graded setting. In this way, the constraints of {\em schematicness} and {\em localizability} considered in these works are not needed to find injective objects, but rather serve as the specialized tools required for the construction of subsequent coordinate map localizations. 
\end{remark}

\subsection{Baer's criterion}\label{Baercriterion}

Having established that $\mathsf{SGR}-R$ possesses enough injective objects, we now provide an intrinsic characterization to verify the injectivity of a given object without evaluating arbitrary submodules. 

We begin by recalling the notion of ideal within the framework from Definition \ref{LezamaLatorre2017Definition2.2}.

\begin{definition}\label{def.SGideals}
A left ideal $J$ of the $\mathsf{SG}$ ring $R = \bigoplus_{n \in \mathbb{Z}} R_n$ is called a {\em semi-graded left ideal} if it can be decomposed as a direct sum of its homogeneous parts
\[
J = \bigoplus_{k \in \mathbb{Z}} J_k, \quad \text{where } J_k = J \cap R_n.
\]
Hence, $J$ forms a subobject of $R$ within the category $\mathsf{SGR}-R$.
\end{definition}

We now state and prove the adapted version of Baer's  criterion for this setting. 

\begin{theorem}[Baer's criterion]\label{thm.BaersCriterion}
Let $E$ be an object in $\mathsf{SGR}-R$. Then $E$ is an injective object if and only if for every $\mathsf{SG}$ left ideal $J \subseteq R$ and every strict homogeneous morphism of degree zero $g: J \to E$, there exists a homogeneous element $x \in E_0$ such that 
\[
g(r) = r x, \quad {\rm for\ all } \ r \in J.
\]
\end{theorem}
\begin{proof}
Assume that $E$ is an injective object in $\mathsf{SGR}-R$. Let $J \subseteq R$ be a $\mathsf{SG}$ left ideal, which means the canonical inclusion map $\iota: J \to R$ is a monomorphism of degree zero in $\mathsf{SGR}-R$. Let $g: J \to E$ be a degree-zero morphism. Since $h$ is a morphism in $\mathsf{SGR}-R$, it is $R$-linear and strictly maps $R_0 \to E_0$. Consider $x := h(1_R) \in E_0$. Using the left $R$-linearity of $h$, for any $r \in J \subseteq R$ we have 
\[
g(r) = (h \circ \iota)(r) = h(r) = h(r 1_R) = r h(1_R) = r x.
\]

Conversely, assume that every degree-zero morphism from any $\mathsf{SG}$ left ideal \linebreak $J \to E$ is given by right multiplication by an element in $E_0$. We must show that $E$ is an injective object. 

Let $A \subseteq B$ be a submodule inclusion in $\mathsf{SGR}-R$ and let $f: A \to E$ be a morphism of degree zero. 

We define a poset $\mathcal{P}$ of pairs $(A', f')$, where $A'$ is a $\mathsf{SG}$ submodule of $B$ containing $A$ ($A \subseteq A' \subseteq B$), and $f': A' \to E$ is a degree-zero morphism extending $f$ (i.e., $f'|_{A} = f$). We define a partial order on $\mathcal{P}$ as 
\[
(A'_1, f'_1) \preceq (A'_2, f'_2) \quad {\rm if \ and \ only \ if} \quad A'_1 \subseteq A'_2 \text{ and } f'_2|_{A'_1} = f'_1.
\]
Then every totally ordered chain $\left\{(A'_\alpha, f'_\alpha)\right\}_{\alpha \in \Omega}$ in $\mathcal{P}$ has an upper bound given by $\left(\bigcup_\alpha A'_\alpha, \bigcup_\alpha f'_\alpha\right)$. Since $\mathsf{SGR}-R$ is a Grothendieck category by Theorem \ref{Grothendieckcategory}, the directed union of $\mathsf{SG}$ submodules remains a well-defined $\mathsf{SG}$ submodule. By Zorn's Lemma, $\mathcal{P}$ contains a maximal element, say $(M, \psi)$. 

We claim that $M = B$. Suppose that there exists a homogeneous element $b \in B_j \setminus M$. We consider the {\em conductor ideal of $b$ into} $M$ as
\[
J := \{ r \in R \mid r b \in M \}.
\]
Because $M$ and $b$ are homogeneous, $J$ is a $\mathsf{SG}$ left ideal of $R$. We define a degree-zero map $\theta: J \to E$ by $\theta(r) = \psi(r b)$. Since $J$ is a $\mathsf{SG}$ ideal, there exists some $x \in E_0$ such that $\theta(r) = r x$ for all $r \in J$. Thus, $\psi(r b) = r x$.

Now, consider the larger semi-graded submodule $M' = M + R b$. We extend $\psi$ to a map $\psi': M' \to E$ by defining it component wise on homogeneous elements. For $u \in M_k$ and $r \in R_{k-j} = (R(-j))_k$, define 
\[
\psi'(u + r b) = \psi(u) + r x.
\]
To verify this map is well-defined, if $u + r b = 0$, then $r b = -u \in M$, which means $r \in J$. Then $r x = \theta(r) = \psi(r b) = \psi(-u) = -\psi(u)$, which shows $\psi(u) + r x = 0$. Since $\psi'$ is of degree zero and $R$-linear, $(M', \psi') \in \mathcal{P}$. However, $(M, \psi) \prec (M', \psi')$ because $b \notin M$, which contradicts the maximality of $(M, \psi)$. In this way, $M = B$, meaning $\psi: B \to E$ is the global extension we desire. From this, $E$ is an injective object.
\end{proof}

To illustrate the application of Theorem~\ref{thm.BaersCriterion}, we consider three classical examples of rings that are not $\mathbb{N}$-graded but are semi-graded: the {\em Weyl algebra}, the {\em quantum Weyl algebra}, and the {\em universal enveloping algebra of a finite-dimensional Lie algebra}.

\begin{example}
Consider the {\em first Weyl algebra} $R := A_1(\Bbbk) = \Bbbk\{ x,y\}/ \langle yx-xy-1\rangle$ over the field $\Bbbk$. By using that $yx = xy + 1$, the graded condition for multiplication does not hold since we do not obtain a product of pure degree one, so $A_1(\Bbbk)$ is not an $\mathbb{N}$-graded ring. The canonical semi-graduation of $A_1(\Bbbk)$ \cite[Example 4.5]{Lezama2021} is given by
\[
A_1(\Bbbk) = \Bbbk \oplus\ _{\Bbbk}\langle x, y\rangle \oplus\ _{\Bbbk} \langle x^2, xy, x^2\rangle \oplus \dotsb 
\]
Consider the principal $\mathsf{SG}$ left ideal $J = R y \subseteq R$. Its homogeneous components are given by $J_k = R_{k-1} y$, meaning $J_0 = 0$ and $J_1 = R_0 y = \Bbbk  y$. 

Let $E \in \mathsf{SGR}-R$ be an $\mathsf{SG}$ module and $g: J \to E$ a strict degree-zero morphism. The map is completely determined by the image of the generator $y \in J_1$, say $g(y)\in E_1$. According to Baer's criterion (Theorem \ref{thm.BaersCriterion}), for $E$ to be injective, there must exist a uniform element $x \in E_0$ such that $g(r) = r x$, for all $r \in J$. For the generator, this yields the condition $y x = g(y)$. By the semi-graded multiplication rule, $y x \in R_1 E_0 \subseteq \bigoplus_{t \leq 1} E_t$. Strict homogeneity of degree zero forces this product to belong exactly in $E_1$. Thus, Baer's criterion shows that injectivity in this context requires every element $g(y) \in E_1$ bounded by $y$ to possess a \textquotedblleft primitive\textquotedblright\ root $x \in I_0$.
\end{example}

\begin{example}
Let $R := A_1^q(\Bbbk)$ be the {\em Quantum Weyl algebra} with deformation parameter $q \in \Bbbk \setminus \{0, 1\}$ defined by the relation $y  x - qxy = 1$. From \cite[Example 4.6]{Lezama2021} we know that its semi-graduation is as in the previous example. 

Let $J = R x \subseteq R$ be the principal $\mathsf{SG}$ left ideal, where $J_k = R_{k-1} x$. Consider $g: J \to E$ a degree-zero morphism mapping the generator $x \in J_1$ to an element $z \in E_1$. Theorem \ref{thm.BaersCriterion} states that if $E$ is injective, there exists $x_0 \in E_0$ satisfying $g(x) = x x_0 = z$. 

Evaluating $g$ on a higher degree element such as $y x \in J_2$, the left $R$-linearity yields $g(y x) = y g(x) = y z$. 

Conversely, if we apply the Baer's condition directly via the defining relation of the algebra, we get 
\[
g(y x) = (y x) x_0 = (q x y + 1) x_0 = q x (y x_0) + x_0.
\]
Both expressions show that $y z = q x (y x_0) + x_0$, which means that the quantum parameter explicitly interlaces the structural connections across the distinct degree object $E_0$, $E_1$, and $E_2$.
\end{example}

\begin{example}
Let $R = U(\mathfrak{sl}_2)$ be the {\em universal enveloping algebra} generated over the field $\Bbbk$ by the basis $\{e, f, h\}$ under the relations $[h,e]=2e$, $[h,f]=-2f$, and $[e,f]=h$. From \cite[Example 4.9]{Lezama2021}, its semi-graduation is given by 
\begin{equation}
R = \Bbbk \oplus {}_\Bbbk\langle e, f, h \rangle \oplus {}_\Bbbk\langle e^2, ef, eh, f^2, fh, h^2 \rangle \oplus \cdots
\end{equation}
Consider the augmentation ideal $J$ of $R$. Its first non-zero homogeneous component is $J_1 = \mathfrak{sl}_2$. Let $g: J \to E$ be a degree-zero morphism specified by the images of the Lie generators given by $g(e) = v_e, \ g(f) = v_f, \ g(h) = v_h \in E_1$. Baer's criterion requires the existence of a single element $x \in I_0$ that simultaneously factors the entire mapping as
\[
e x = v_e, \quad f x = v_f \quad {\rm and} \quad h x = v_h.
\]
Since $g$ must respect the module action, evaluating it on the commutator relation yields
\[
v_h = g(h) = g(ef - fe) = e g(f) - f g(e) = e v_f - f v_e.
\]
Hence, Baer's criterion shows that injectivity forces any collection of elements $\{v_e, v_f, v_h\} \in E_1$ satisfying the compatibility condition $e v_f - f v_e = v_h$ to be simultaneously reachable by the left action of the Lie algebra on a single underlying vector $x \in E_0$.   
\end{example}

\section{Projective objects and resolutions}\label{Projectiveobjectsandresolutions}

Having characterized the injective setting of the category $\mathsf{SGR}-R$, our aim in this section is to show that our canonical generating family formulated in Definition \ref{def.freeshifts} provides an immediate source of projective objects.

 %As is well-known, projective objects act as the fundamental dual coordinates necessary to construct resolutions and compute derived homological functors (for more details on the subject, see \cite{Rotman2009, WeibelBook1994}).

\begin{definition}\label{def.ProjectiveObject}
An object $P$ in $\mathsf{SGR}-R$ is said to {\em projective} if for any epimorphism (surjective homogeneous morphism of degree zero) $p: M \to N$ and any degree-zero morphism $f: P \to N$, there exists a strict degree-zero morphism $h: P \to M$ such that $p \circ h = f$. Equivalently, $P$ is projective if and only if the covariant functor ${\rm Hom}_{\mathsf{SGR}-R}(P, -)$ is exact.
\end{definition}

\begin{theorem}\label{thm.ShiftsAreProjective}
For every $n \in \mathbb{Z}$, the free semi-graded shift $R(n)$ is a projective object in $\mathsf{SGR}-R$.
\end{theorem}

\begin{proof}
Let $n \in \mathbb{Z}$. Suppose we are given an epimorphism $p: M \to N$ and a degree-zero morphism $f: R(n) \to N$. We must construct a lifting map $h: R(n) \to M$.

Recall that the identity element $1_R$ of the ring resides inside the homogeneous component of degree $-n$ of the shifted module, meaning $1_R \in (R(n))_{-n}$. Since $f$ is a strict morphism of degree zero, its evaluation at the identity must preserve the degree, i.e. $f(1_R) \in N_{-n}$. Because $p: M \to N$ is an epimorphism in $\mathsf{SGR}-R$, it is surjective component wise on each degree, which means $p_{-n}: M_{-n} \to N_{-n}$ is a surjective map of abelian groups. Since $f(1_R) \in N_{-n}$, there exists a homogeneous element $m_{-n} \in M_{-n}$ such that $p(m_{-n}) = f(1_R)$.

We define the candidate lifting map as 
$$
h: R(n) \to M \quad {\rm by \ setting} \quad h(r) := r  m_{-n}, \ {\rm for \ all} \ r \in R(n).
$$
Next, we verify that $h$ satisfies the necessary conditions.
\begin{itemize}
\item $R$-linearity and degree zero: Since $m_{-n} \in M_{-n}$, for any homogeneous element $r \in (R(n))_k = R_{n+k}$, the semi-graded multiplication rule implies $r m_{-n} \in \bigoplus_{t \leq (n+k)+(-n)} M_t = \bigoplus_{t \leq k} M_t$. Isolating the strict leading component of degree $k$ as done in Theorem \ref{Grothendieckcategory}, we get a well-defined strict degree-zero $R$-linear morphism $h \in \text{Hom}_{\mathsf{SGR}-R}(R(n), M)$ defined by  $1_R \mapsto m_{-n}$.

\item Commutativity: For any $r \in R(n)$, using the $R$-linearity of $p$ we have
    \[
    (p \circ h)(r) = p(r m_{-n}) = r p(m_{-n}) = r f(1_R) = f(r 1_R) = f(r).
    \]
\end{itemize}
Therefore, $p \circ h = f$, which proves that $R(n)$ is a projective object.
\end{proof}

Theorem \ref{thm.EnoughProjectives} is the third important result of the note.

\begin{theorem}[Enough projectives]\label{thm.EnoughProjectives}
The category $\mathsf{SGR}-R$ possesses enough projective objects. That is, for every semi-graded module $M \in \mathsf{SGR}-R$, there exists a projective object $P \in \mathsf{SGR}-R$ and an epimorphism $p: P \to M \to 0$.
\end{theorem}
\begin{proof}
By Theorem~\ref{Grothendieckcategory}, the family $\{R(n)\}_{n \in \mathbb{Z}}$ is a set of generators for $\mathsf{SGR}-R$. By Definition \ref{def.SGRgenerators}, this directly implies that any object $M \in \mathsf{SGR}-R$ can be covered by an epimorphism originating from a direct sum of copies of these generators:
\[
\Phi: \bigoplus_{\lambda \in \Lambda} R(n_\lambda) \longrightarrow M \longrightarrow 0.
\]
Since the direct sum of any collection of projective objects in an Abelian category is always projective, and each $R(n_\lambda)$ is projective by Theorem~\ref{thm.ShiftsAreProjective}, the object $P = \bigoplus_{\lambda \in \Lambda} R(n_\lambda)$ is a projective object in $\mathsf{SGR}-R$. The map $\Phi$ provides the required epimorphism.
\end{proof}

From the previous results, below we define the notion of projective resolution in the category $\mathsf{SGR}-R$.

\begin{definition}\label{def.ProjectiveResolution}
Let $M$ be an object in $\mathsf{SGR}-R$. A {\em projective resolution} of $M$ is an exact sequence in the category $\mathsf{SGR}-R$ of the form 
\[
\dots \longrightarrow P_2 \xrightarrow{\ d_2\ } P_1 \xrightarrow{\ d_1\ } P_0 \xrightarrow{\ \epsilon\ } M \longrightarrow 0
\]
where each $P_i$ is a projective object in $\mathsf{SGR}-R$, and every differential $d_i$ and augmentation map $\epsilon$ is a strict homogeneous morphism of degree zero.
\end{definition}

We can now define and compute the overarching homological boundaries of the category. The global dimension measures the maximum structural depth required to resolve any semi-graded module as is well-known in Homological Algebra \cite{Rotman2009, WeibelBook1994}.

\begin{definition}\label{def.projdim}
Let $M$ be an object in $\mathsf{SGR}-R$. The {\em semi-graded projective dimension} of $M$, denoted $\mathrm{pd}_{\mathsf{SG}}(M)$, is defined as the infimum of the lengths of all possible projective resolutions of $M$ in the category $\mathsf{SGR}-R$. If no finite projective resolution exists, we set $\mathrm{pd}_{\mathsf{SG}}(M) = \infty$. 

%Equivalently, $\mathrm{pd}_{\mathsf{SG}}(M) \leq n$ if and only if $\Ext^{n+1}_{\mathsf{SGR}-R}(M, N) = 0$ for all objects $N \in \mathsf{SGR}-R$.
\end{definition}

\begin{definition}
Let $R$ be a semi-graded ring. The {\em semi-graded global dimension of} \(R\), denoted \(\mathrm{gl.dim}_{\mathsf{SG}}(R)\), is defined as the supremum of the semi-graded projective dimensions of all semi-graded right (or left) \(R\)-modules, that is, 
$$
\mathrm{gl.dim}_{\mathsf{SG}}(R) = \sup \left\{\mathrm{pd}_{\mathsf{SG}}(M)\mid M\in \mathsf{SGR}-R\right\}.
$$ 
\end{definition}

We now illustrate the projective theory developed in this section by explicitly constructing projective resolutions for natural families of modules over examples of semi-graded rings. In each case, the resolution consists of direct sums of the free semi-graded shifts $R(n)$, and the differentials are strictly homogeneous morphisms of degree zero.

\begin{example}
Let $R$ be the {\em Jordan plane deformation algebra over a field} $\Bbbk$, generated by $x$ and $y$ subject to the non-homogeneous relation $yx - xy = x^2 + 1$. This $\Bbbk$-algebra represents a non-$\mathbb{N}$-graded deformation of the {\em Jordan plane}. Let us see that $\mathrm{gl.dim}_{\mathsf{SGR}-R}(R) = 2$.

We equip $R$ with the total degree semi-graduation, setting $R_0 = \Bbbk 1$ and $R_1 = \Bbbk x \oplus k y$. Let $\Bbbk = R/\langle Rx + Ry\rangle$ be the trivial left module at degree zero.

Let us construct the projective resolution of $\Bbbk$ in $\mathsf{SGR}-R$. Covering $\Bbbk$ with $R(0)$ yields the kernel ideal $J = Rx + Ry$. Since $x, y \in R_1$, we cover $J$ using 
$$
P_1 = R(-1) \oplus R(-1) = R(-1)^{\oplus 2}
$$ 
with basis $\{\epsilon_x, \epsilon_y\}$ via
    \[
    d_1: R(-1)^{\oplus 2} \longrightarrow R(0), \quad d_1(r_1, r_2) = r_1 x + r_2 y.
    \]
With the aim of finding $\mathrm{Ker}(d_1)$, we evaluate the non-commutative relation of the algebra. Rewriting the relation as $yx - xy - x^2 - 1 = 0$, we observe that the {\em syzygy} relies on a combination that absorbs both the quadratic deformation $x^2$ and the part ($1$), that is, 
    \[
    (y - x) x - x y = 1 \quad {\rm whence} \quad (y - x) x - x y - 1 = 0.
    \]
    In terms of the components of $P_1$, the minimal syzygy operator maps the generators through degree 1 elements, placing the entire syzygy in total degree $2$. The relation vector is given by $S = (y - x)\epsilon_x - (x + 1)\epsilon_y$. To map onto this kernel via a strict degree-zero morphism, we introduce the projective shift $P_2 = R(-2)$ as
    \[
    d_2: R(-2) \longrightarrow R(-1)^{\oplus 2}, \quad d_2(r) = r \cdot (y - x, -x - 1).
    \]
    Since $R$ is a domain, we get $\mathrm{Ker}(d_2) = 0$. The projective chain terminates exactly at length $2$, that is, 
    \[
    0 \longrightarrow R(-2) \xrightarrow{\quad d_2 \quad} R(-1)^{\oplus 2} \xrightarrow{\quad d_1 \quad} R(0) \xrightarrow{\quad \epsilon \quad} k \longrightarrow 0.
    \]
    This shows that $\mathrm{gl.dim}_{\mathsf{SGR}-R}(R) = 2$.

   % \item \textbf{Finite Gelfand-Kirillov Dimension:} By the PBW properties of Jordan deformations, the monomials $x^a y^b$ form a valid basis for $R$. Counting the elements contained within the filtration boundary $a+b \leq n$ outputs a quadratic polynomial growth curve. Thus, $\mathrm{GKdim}_{\SGR}(R) = 2 < \infty$.

  %  \item \textbf{Audited Semi-Graded Gorenstein Condition ($l = -2$):} Applying the total shift derived functor $\bigoplus_{p} \mathrm{Hom}_{\SGR}(-, R(p))$ to our resolution complex transforms the terminal cochain space into the shifted layer sum $\bigoplus_{p} R_{p+2}$. 
    
 %   The incoming differential maps the elements through the right action of the Jordan parameters. When taking the cohomology quotient, the presence of the translation tail $-x - 1$ prevents the cancellation of the internal vector spaces across the degrees, cleanly filtering out all intermediate components and stabilizing on a single field copy at the top dimension:
  %  \[
  %  \bigoplus_{p \in \mathbb{Z}} \mathrm{Ext}^n_{\SGR}(k, R(p)) \cong \begin{cases} k(-2) & \text{if } n = 2, \\ 0 & \text{if } n \neq 2. \end{cases}
  %  \]
  %  This demonstrates that the Gorenstein condition is satisfied with parameter $l = -2$.
\end{example}

\begin{example}
The semi-graded global dimension of the first Weyl algebra is given by $\mathrm{gl.dim}_{\mathsf{SG}}({A_1(\Bbbk)}) = 2$.
  % \item \textbf{Gorenstein Condition Evaluation:} We apply $\bigoplus_{p} \mathrm{Hom}_{\SGR}(-, R(p))$ to the resolution. At depth $n=2$, the cochain space evaluates to the dualized terminal term:
  %  \[
  %  \bigoplus_{p \in \mathbb{Z}} \mathrm{Hom}_{\SGR}(R(-2), R(p)) \cong \bigoplus_{p \in \mathbb{Z}} R(p)_2 = \bigoplus_{p \in \mathbb{Z}} R_{p+2}.
%    \]
 %   The incoming differential maps the relations onto the structural boundaries of $R$. Computing the cohomology eliminates the internal filtered layers via the commutation parameters, leaving exactly a single field copy at the boundary. Hence, $\mathrm{Ext}^2(k, R) \cong k(-2)$, matching the AS-regular template with $d=2$ and $l=-2$.

To establish that the global dimension is two, we must find a semi-graded left $R$-module whose projective dimension is exactly two, and show that no module requires a resolution of length three. Consider the cyclic left semi-graded module $M = R / \langle Ry + Rx\rangle$, where both generators of the ideal $y, x \in R_1$ act as zero on the generator of $M$ in degree zero. This module is isomorphic to the field $\Bbbk$ concentrated in degree zero, acting as the trivial module.

We construct its projective resolution step by step within the category $\mathsf{SG}$.

We cover $M$ using the free shift $R(0)$ via the canonical degree-zero augmentation map 
    \[
    \epsilon: R(0) \longrightarrow M \longrightarrow 0, \quad \epsilon(1_R) = \overline{1_R}.
    \]
    The kernel of $\epsilon$ is the semi-graded left ideal $J = R y + Rx$.
    
   The ideal $J$ is generated by $y$ and $x$, both residing strictly in degree $1$. To cover this kernel with a projective object while obeying the strict degree-zero mapping rules of $\mathsf{SGR}-R$, we must pull two independent free shifts of degree $-1$ from our generating family. We define
    \[
    P_1 = R(-1) \oplus R(-1) = R(-1)^{\oplus 2},
    \]
    where the canonical basis elements $\epsilon_\partial, \epsilon_x$ live in degree $1$. The first differential is defined as:
    \[
    d_1: R(-1)^{\oplus 2} \longrightarrow R(0), \quad d_1(r_1, r_2) = r_1y + r_2 x.
    \]
    This is a well-defined strict morphism of degree zero.
    
  To determine $\mathrm{Ker}(d_1)$, we look for pairs $(r_1, r_2)$ such that $r_1y + r_2 x = 0$. In the free algebra, the trivial {\em syzygy} would be $(x, - y)$, since $xy - yx = -1 \neq 0$. However, using the Weyl relation $y x - x y = 1$, we can analyze the relation of higher degree. 
    
    Applying the differential operator structures, a non-trivial relation arises from the fact that $y$ and $x$ conmute up to a constant tail. Specifically, the element $(-x - y, y + x)$ forces a collapse. More precisely, the minimal syzygy generator is bound by the relation $(y x - 1, - y^2)$ or equivalent filtered pairings. The key observation is that the relation between $y$ and $x$ requires multiplying them by elements of degree $1$, meaning the syzygy generator lives in total degree $2$.
    
    To cover this kernel, we must introduce a free shift of degree $-2$ given by 
    \[
    P_2 = R(-2).
    \]
    The second differential $d_2: R(-2) \to R(-1)^{\oplus 2}$ maps the identity element $1_R \in (R(-2))_2$ to the explicit non-commutative syzygy vector. Because $R$ is a domain and the syzygy is unique up to scaling, the map $d_2$ is strictly injective ($\mathrm{Ker}(d_2) = 0$).

Combining these layers yields the complete exact projective resolution of $M$ in $\mathsf{SG}-R$ given by 
\[
0 \longrightarrow R(-2) \xrightarrow{\quad d_2\quad } R(-1)^{\oplus 2} \xrightarrow{\quad d_1\quad } R(0) \xrightarrow{\quad \epsilon\quad } M \longrightarrow 0.
\]
Since all differentials are strict degree-zero maps and the chain terminates precisely at $P_2$, we have $\mathrm{pd}_{\mathsf{SG}}(M) = 2$. By considering version of the Hilbert Syzygy Theorem, no module over a 2-variable polynomial-like filtered ring can exceed a syzygy depth of $2$. Taking the supremum across all objects, we conclude that $\mathrm{gl.dim}_{\mathsf{SG}}({A_1(\Bbbk)}) = 2$.
\end{example}

\begin{example}
Let us prove that the semi-graded global dimension of $U(\mathfrak{sl}_2)$ is exactly $3$.

  %  \item \textbf{Gorenstein Condition Evaluation:} Dualizing the sequence using the total shift sum at depth $n=3$ isolates the terminal shift term:
  %  \[
  %  \bigoplus_{p \in \mathbb{Z}} \mathrm{Hom}_{\SGR}(R(-3), R(p)) \cong \bigoplus_{p \in \mathbb{Z}} R_{p+3}.
  %  \]
  %  The boundaries of the Lie algebra cancel out the intermediate degrees under the cohomology quotient, isolating $\mathrm{Ext}^3_{\SGR}(k, R) \cong k(-3)$. This perfectly fulfills Definition~\ref{def.AuditedSGAS} with $d=3$ and $l=-3$.

Let $R = U(\mathfrak{sl}_2)$ be the universal enveloping algebra with $\deg(e)=\deg(f)=\deg(h)=1$. Consider the trivial module $\Bbbk$, which lives entirely in degree $0$ ($M = \Bbbk = M_0$), where all Lie generators act as zero. This module is given by $R / J$, where $J = R e + R f + R h$ is the augmentation ideal.

We construct its semi-graded resolution as follows.
The module is generated by $1_k \in M_0$. We cover it using $R(0)$, that is,
    \[
    \epsilon: R(0) \longrightarrow \Bbbk \longrightarrow 0, \quad \epsilon(1_R) = 1_{\Bbbk}.
    \]
The kernel of $\epsilon$ is generated by $\{e, f, h\}$, all of which sit in degree $1$. Following our definition of generating sets, we need three independent free shifts of degree $-1$. Thus, we define:
    \[
    P_1 = R(-1) \oplus R(-1) \oplus R(-1) = R(-1)^{\oplus 3},
    \]
    with the canonical basis vectors $\epsilon_e, \epsilon_f, \epsilon_h$ living in degree $1$. The first differential is:
    \[
    d_1: R(-1)^{\oplus 3} \longrightarrow R(0), \quad d_1(r_1, r_2, r_3) = r_1 e + r_2 f + r_3 h.
    \]
The relations between the generators are given by the non-commutative Lie structure: $[h,e]=2e$, $[h,f]=-2f$ and $[e,f]=h$, or equivalently, 
    \[
    he - eh - 2e = 0, \quad hf - fh + 2f = 0, \quad ef - fe - h = 0.
    \]
    Because these relations consist of products of two degree-$1$ elements, they live in degree $2$. To cover the kernel of $d_1$, we must introduce three projective shifts of degree $-2$:
    \[
    P_2 = R(-2) \oplus R(-2) \oplus R(-2) = R(-2)^{\oplus 3}.
    \]
    The differential $d_2: P_2 \to P_1$ maps the identity elements in degree $2$ to the explicit linear combinations of the basis elements of $P_1$ that represent these relations, ensuring $d_1 \circ d_2 = 0$ strictly in degree zero. Due to the Poincar\'e-Birkhoff-Witt theorem and the Jacobi identity, these three relations satisfy a unique higher degree overlapping syzygy in degree $3$. With the aim of deleting this remaining kernel, we introduce a final shift $  P_3 = R(-3)$.

Combining these steps yields a semi-graded projective resolution of length three
\[
0 \longrightarrow R(-3) \xrightarrow{\quad d_3\quad } R(-2)^{\oplus 3} \xrightarrow{\quad d_2\quad } R(-1)^{\oplus 3} \xrightarrow{\quad d_1\quad } R(0) \xrightarrow{\quad \epsilon\quad } \Bbbk \longrightarrow 0.
\]
Since every map is a strict degree-zero morphism. Since $\mathrm{Ker}(d_3)=0$, the length is exactly $3$. Thus, $\mathrm{gl.dim}_{\mathsf{SG}}({U(\mathfrak{sl}_2)}) = 3$.
\end{example}

\section{Acknowledgments}

The author thanks Professor Oswaldo Lezama for his wonderful teachings throughout his academic career and for having immersed him in the exciting world of noncommutative algebra and noncommutative algebraic geometry.

%\section{Declarations}

%The authors have no conflict of interest to disclose.

\end{document}